\documentclass[oneside,english]{amsart}
\usepackage[T1]{fontenc}
\usepackage[latin9]{inputenc}
\usepackage{babel}
\usepackage{array}
\usepackage{verbatim}
\usepackage{float}
\usepackage{url}
\usepackage{multirow}
\usepackage{amsthm}
\usepackage{amssymb}
\usepackage{stmaryrd}
\usepackage[all]{xy}
\usepackage[unicode=true,
 bookmarks=true,bookmarksnumbered=false,bookmarksopen=false,
 breaklinks=false,pdfborder={0 0 1},backref=false,colorlinks=false]
 {hyperref}
\hypersetup{
 pdfauthor={Giosu{\`e} Muratore}}

\makeatletter

\providecommand{\tabularnewline}{\\}

\numberwithin{equation}{section}
\numberwithin{figure}{section}
\numberwithin{table}{section}
\theoremstyle{plain}
\newtheorem{thm}{\protect\theoremname}[section]
\theoremstyle{definition}
\newtheorem{defn}[thm]{\protect\definitionname}
\theoremstyle{remark}
\newtheorem{rem}[thm]{\protect\remarkname}
\theoremstyle{plain}
\newtheorem{prop}[thm]{\protect\propositionname}
\theoremstyle{plain}
\newtheorem{cor}[thm]{\protect\corollaryname}
\theoremstyle{definition}
\newtheorem{example}[thm]{\protect\examplename}

\usepackage{tikz}
\usetikzlibrary{positioning}

\makeatother

\providecommand{\corollaryname}{Corollary}
\providecommand{\definitionname}{Definition}
\providecommand{\examplename}{Example}
\providecommand{\propositionname}{Proposition}
\providecommand{\remarkname}{Remark}
\providecommand{\theoremname}{Theorem}

\begin{document}
\global\long\def\ev{\mathrm{ev}}%
\global\long\def\d{\mathrm{d}}%
\global\long\def\MC{\mathcal{S}}%
\global\long\def\gra{\Gamma(m,d)}%
\global\long\def\grapos{\Gamma(m,d)^{+}}%
\global\long\def\Sp{\mathrm{Sp}}%

\title{Irreducible Contact Curves via Graph Stratification}
\author{Giosu{\`e} Muratore}
\date{\today}
\address{CMAFcIO, Faculdade de Ci\^{e}ncias da ULisboa, Campo Grande 1749-016 Lisboa,
Portugal}
\email{\href{mailto:muratore.g.e@gmail.com}{muratore.g.e@gmail.com}}
\urladdr{\url{https://sites.google.com/view/giosue-muratore}}
\begin{abstract}
We prove that the moduli space of contact stable maps to $\mathbb{P}^{2n+1}$
of degree $d$ admits a stratification parameterized by graphs. We
use it to determine the number of irreducible rational contact curves
in $\mathbb{P}^{2n+1}$ with any Schubert condition. We give explicitely some of 
these invariants for $\mathbb{P}^{3}$ and $\mathbb{P}^{5}$. We give another proof of the formula for the number of plane contact curves in $\mathbb{P}^{3}$ meeting the appropriate number of lines.
\end{abstract}

\maketitle

\section{Introduction}

A fast-growing area of mathematics is Contact Geometry, that is the
study of manifold with a contact structure. First defined by \cite{kobayashi1959remarks},
a contact structure on a smooth complex manifold $X$ of odd dimension
is a corank $1$ non-integrable holomorphic distribution. All projective
spaces of odd dimension admit a (unique) contact structure. 

The interest for these manifolds comes from real geometry. Calabi
\cite{C2} proved that the study of harmonic maps of spheres $f\colon S^{2}\rightarrow S^{4}$
is equivalent, modulo an involution, to the study of holomorphic
horizontal curves $\tilde{f}\colon\mathbb{CP}^{1}\rightarrow\mathbb{CP}^{3}$
of fixed degree. Roughly speaking, $f$ minimizes the energy functional if
$f$ or $-f$ has a lift $\tilde{f}$, such that $\tilde{f}(\mathbb{P}^{1})$
is integral with respect to the contact structure of $\mathbb{P}^{3}$.

% Contact manifolds have also very interesting applications in holonomy
% theory (see \cite{Bea99}) and Minimal model program (see \cite{druel1998structures,kebekus2000projective}).

The enumeration of contact curves (i.e., rationally connected curves integral
with respect to the contact structure) with geometric conditions started
in \cite{levcovitz2011symplectic}. Levcovitz and Vainsencher found,
among other results, the number of contact plane curves of degree
$d$ meeting $3+d$ general lines in $\mathbb{P}^{3}$. On the other
hand, the first systematic study of enumerative invariants of those
curves started in \cite{Mur} (strengthening results from \cite{Eden}).
The strategy was to consider the closed subscheme $\MC_{m}(\mathbb{P}^{n},d)\subset\overline{\mathcal{M}}_{0,m}(\mathbb{P}^{n},d)$
of stable maps whose image is a contact curve. Certain Gromov-Witten
invariants related to $\MC_{m}(\mathbb{P}^{n},d)$ are enumerative.
Using localization, the author was able to give enumerative invariants
of contact curves with any Schubert condition.

The motivation to this paper is that very few enumerative invariants
are known for \textbf{irreducible} rational contact curves. The reason
is that $\MC_{m}(\mathbb{P}^{n},d)$ has many irreducible components.
As already pointed out in \cite[Remark 3.10]{Mur}, the subscheme
of $\MC_{m}(\mathbb{P}^{n},d)$ parameterizing reducible curves has
codimension $0$. That is, the ``boundary'' of $\MC_{m}(\mathbb{P}^{n},d)$
is not a divisor, but a component purely of the same dimension. So
that, the enumerative invariants get contribution from all irreducible
components.%
\begin{comment}
For example, the moduli space of horizontal morphisms $\mathbb{P}^{1}\rightarrow\mathbb{P}^{3}$
of degree $d$ has two irreducible components if $d\ge3$ \cite{Loo,KL}.
\end{comment}

In this paper we give a decomposition of $\MC_{m}(\mathbb{P}^{n},d)$
in components parameterized by graphs, with a complete description
of the dimension of all components (Theorem \ref{thm:codim}). This
stratification is induced by a similar one of $\overline{\mathcal{M}}_{0,m}(\mathbb{P}^{n},d)$.
We use this stratification to give an explicit formula to compute
the number of reducible contact curves with any dual graph (Corollary
\ref{cor:int_prod}). Moreover, we are able to give another proof
of Levcovitz and Vainsencher's result (Proposition \ref{prop:LV}).

In Tables \ref{tab:EnuNum} and \ref{tab:EnuNuminP5-con}, we give
enumerative invariants of irreducible rational contact curves in $\mathbb{P}^{3}$
and $\mathbb{P}^{5}$ of low degree. The strategy is to subtract,
from the number of contact curves, the number of reducibles contact
curves with the same Schubert conditions. Note that the enumerative
numbers of contact curves in $\mathbb{P}^{2n+1}$ are all known thanks
to \cite{MurSch}.

An alternative way to get enumerative invariants of irreducible contact
curves may be to use another moduli space, like the Hilbert scheme.
We plan to pursue this research path in the future.

I would like to thank Carlos Florentino for his many suggestions.
This problem came out from one of the many conversations I had with
Israel Vainsencher during my work at UFMG. I would like to thank him
for his time. I also thank Giordano Cotti, Angelo Lopez and Filippo
Viviani for useful conversations. Finally, I thank Csaba Schneider
for his computational support.

The author is supported by FCT - Funda\c{c}\~{a}o para a Ci\^{e}ncia e a Tecnologia,
under the project: UIDP/04561/2020. The author is a member of GNSAGA (INdAM).

\section{Contact Structures}

All varieties in this article are defined over $\mathbb{C}$. In this
section $n$ is any non-negative integer.
\begin{defn}
Let $X$ be a $(2n+1)$-dimensional complex manifold. A contact structure
on $X$ is an open cover $\{U_{i}\}_{i}$ of $X$ together with $1$-forms
$\alpha_{i}$ on each $U_{i}$ such that:
\begin{enumerate}
\item at every point of $U_{i}$, the form $\alpha_{i}\wedge(\d\alpha_{i})^{\wedge n}$
is non-zero, and
\item if the intersection $U_{i}\cap U_{j}$ is non empty, then there exists
a non-vanishing holomorphic function $f_{ij}$ on $U_{i}\cap U_{j}$
such that $\alpha_{i}=f_{ij}\alpha_{j}$ on $U_{i}\cap U_{j}$.
\end{enumerate}
A contact curve is a rationally connected curve $C\subset X$ such that any local $1$-form
$\alpha_{i}$ vanishes at every smooth point of $C\cap U_{i}$.
\end{defn}

\begin{rem}
An equivalent definition is the following: A contact structure on $X$ is a pair $(X,L)$ where $L$ is a line
subbundle of $\Omega_X^1$ such that if $s$ is a non trivial local section of $L$, then $s\wedge  (\d s)^{\wedge n}$
is everywhere non-zero.
\end{rem}

Any symplectic $1$-form on $\mathbb{C}^{2n+2}$ defines a contact
structure on $\mathbb{P}^{2n+1}$ by the projection $\mathbb{C}^{2n+2}\backslash\{0\}\rightarrow\mathbb{P}^{2n+1}$.
Vice versa any contact structure on $\mathbb{P}^{2n+1}$ is induced by a symplectic form on $\mathbb{C}^{2n+2}$. Let $\{x_{0},\ldots,x_{n},y_{0},\ldots,y_{n}\}$
be local coordinates of $\mathbb{C}^{2n+2}$. All symplectic $1$-forms
on $\mathbb{C}^{2n+2}$ are conjugated to

\[
\alpha:=\sum_{i=0}^{n}x_{i}\d y_{i}-y_{i}\d x_{i}.
\]
In particular, there exists a unique contact structure on $\mathbb{P}^{2n+1}$
modulo conjugation. See \cite[1.4.2]{okonek1980vector}.

\begin{example}
On
$\mathbb P^3$, the distribution induced by $\alpha$ assigns to 
$p=[a:b:c:d]$ the plane with equation
$$H_p:=\{[X:Y:Z:W]\in\mathbb P^3 \slash bX-aY-dZ+cW=0\}.$$
A curve is said to be a contact
curve if the tangent line at each smooth point $p$
is contained in the plane $H_p$.
The contact curves of degree $1$ in $\mathbb P^{2n+1}$ are parameterized by the symplectic Grassmannian of isotropic $2$-spaces of $\mathbb C^{2n+2}$. See \cite{levcovitz2011symplectic} for other examples.
\end{example}
\begin{rem}
\label{rem:contact_local}We do not require that a contact
curve $C$ is irreducible. Since the property
of being contact is local, every component of $C$ is contact
in its own.
\end{rem}

\subsection{Symplectic group}

\begin{comment}
For every matrix $M$, we denote by $M^{T}$ the transpose of $M$,
and by $M^{-1}$ the inverse. We also denote by $M^{-T}$ the inverse
of the transpose of $M$ (or the transpose of the inverse, as there
is no ambiguity). We denote by $I$ the unit matrix, and by $\delta_{i,j}$
its coefficients (i.e., $\delta_{i,j}$ is the Kronecker delta).

Let $n$ be any non-negative integer. 
\end{comment}

Let us denote by $\Sp(2n+2,\mathbb{C})$ the set of square matrices
$M$ of order $2n+2$ with complex coefficients such that $M^{T}\Omega M=\Omega$,
where $\Omega$ is the block matrix 
\[
\Omega=\left(\begin{array}{cc}
0 & I\\
-I & 0
\end{array}\right),
\]
and $I$ is the unit matrix of order $n+1$. The following facts are
well-known:
\begin{itemize}
\item The set $\Sp(2n+2,\mathbb{C})$ is a subgroup of the group of matrices
with determinant $1$. The two groups coincide if and only if $n=0$.
\item The standard action of $\Sp(2n+2,\mathbb{C})$ on $\mathbb{C}^{2n+2}$
is transitive, and preserves the $1$-form $\alpha$.
\end{itemize}
There exists an explicit description of $\Sp(2n+2,\mathbb{C})$. Let
$M$ be a block matrix 
\[
M=\begin{pmatrix}A & B\\
C & D
\end{pmatrix},
\]
where $A,B,C,D$ are matrices of order $n+1$. Then $M$ is in $\Sp(2n+2,\mathbb{C})$
if and only if
\begin{eqnarray}
C^{T}A & = & A^{T}C\label{eq:CA}\\
A^{T}D-C^{T}B & = & I\label{eq:AD}\\
D^{T}B & = & B^{T}D.\label{eq:DB}
\end{eqnarray}

Main references are \cite[Lecture 16]{FH} and \cite{Omeara}. It
follows that the standard action of $\Sp(2n+2,\mathbb{C})$ on $\mathbb{P}^{2n+1}$
preserves the contact structure given by $\alpha$. More generally,
for any contact structure on $\mathbb{P}^{2n+1}$ there is an action
of $\Sp(2n+2,\mathbb{C})$ preserving it. This follows from the fact
that all contact structures are equivalent by a change of coordinates.

Let $\{U_{\imath}\}_{\imath=0}^{n}$ be the standard open cover of
$\mathbb{P}^{n}$. It is well known that there exist subgroups $\{H_{\imath}\}_{\imath=0}^{n}$
of $\mathrm{GL}(n+1,\mathbb{C})$ isomorphic to $\mathbb{C}^{n}$,
such that the action $H_{\imath}\curvearrowright U_{\imath}$ is the
translation (see \cite[1.5]{MR507725}). 
% Such subgroups are in the form $H_{\imath}=\{I+[a]\slash a\},$
% where $a=(a_{0},a_{1},\ldots,a_{\imath-1},a_{\imath}=0,a_{\imath+1},\ldots,a_{n})\in\mathbb{C}^{\imath}\times\{0\}\times\mathbb{C}^{n-\imath}$
% and $[a]$ is the matrix whose coefficients are $[a]_{k,j}=\delta_{\imath,j}a_{k}$ ($\delta_{\imath,j}$ is the Kronecker delta).
Unfortunately, such subgroups do not preserve the contact form. The
following proposition may be known to experts, but I could not find any
reference.
\begin{prop}
\label{prop:action_Sp}Consider the standard action of $\Sp(2n+2,\mathbb{C})$
on $\mathbb{P}^{2n+1}$. For every $\imath=0,\ldots,2n+1$ there exist
\begin{itemize}
\item a subgroup $H_{\imath}<\Sp(2n+2,\mathbb{C})$,
\item an open subset $U_{\imath}\subset\mathbb{P}^{2n+1}$, and
\item two isomorphisms $\varphi_{\imath}\colon\mathbb{C}^{2n+1}\rightarrow H_{\imath}$
and $\phi_{\imath}\colon\mathbb{C}^{2n+1}\rightarrow U_{\imath}$,
\end{itemize}
such that
\begin{enumerate}
\item $U_{\imath}$ is $H_{\imath}$-invariant,
\item the action of $H_{\imath}$ on $U_{\imath}$ is free and transitive,
and
\item the open subsets $\{U_{\imath}\}_{\imath=0}^{2n+1}$ form an open
cover of $\mathbb{P}^{2n+1}$.
\end{enumerate}
\end{prop}

\begin{proof}
Let $\{x_{0},\ldots,x_{n},y_{0},\ldots,y_{n}\}$ be a basis of $\mathbb{C}^{2n+2}$,
and $\imath\in\{0,\ldots,n\}$. Let $U_{\imath}$ be the open subset
of $\mathbb{P}^{2n+1}$ of vectors with coefficient of $x_{\imath}$
non-zero, and let $\phi_{\imath}\colon\mathbb{C}^{2n+1}\rightarrow U_{\imath}$
be the standard isomorphism.

Let us define the map $\varphi_{\imath}\colon\mathbb{C}^{2n+1}\rightarrow\Sp(2n+2,\mathbb{C})$
in the following way. For every $(\vec{a},\vec{c})=(a_{0},a_{1},\ldots,\widehat{a_{\imath}},\ldots,a_{n},c_{0},c_{1},\ldots,c_{n})\in\mathbb{C}^{2n+1}$,
$\varphi_{\imath}(\vec{a},\vec{c})$ is the linear map sending
\[
\begin{array}{ccl}
x_{j} & \mapsto & x_{j}+c_{j}y_{\imath}\\
x_{\imath} & \mapsto & x_{\imath}+c_{\imath}y_{\imath}+\sum_{k\neq\imath}a_{k}x_{k}+c_{k}y_{k}\\
y_{j} & \mapsto & y_{j}-a_{j}y_{\imath}\\
y_{\imath} & \mapsto & y_{\imath}
\end{array}
\]
where $j=0,\ldots,\hat{\imath},\ldots,n$. It is an easy exercise
to prove that $\varphi_{\imath}(\vec{a},\vec{c})$ defines a linear
map, whose inverse is $\varphi_{\imath}(-\vec{a},-\vec{c})$ and
$\varphi_{\imath}(\vec{a},\vec{c})\varphi_{\imath}(\vec{b},\vec{d})=\varphi_{\imath}(\vec{q},\vec{w})$
where
\begin{equation}
\begin{array}{ccl}
q_{j} & = & a_{j}+b_{j}\\
w_{j} & = & c_{j}+d_{j}\\
w_{\imath} & = & c_{\imath}+d_{\imath}+\sum_{k\neq\imath}(c_{k}b_{k}-a_{k}d_{k}).
\end{array}\label{eq:q,w}
\end{equation}
By direct computation, $\varphi_{\imath}(\vec{a},\vec{c})$ fixes
the $1$-form $\sum_{k=0}^{n}x_{k}\d y_{k}-y_{k}\d x_{k}$ (see also
Remark \ref{rem:Matrix_form}). Finally the coefficient of $x_{\imath}$
is preserved, hence $U_{\imath}$ is $\varphi_{\imath}(\vec{a},\vec{c})$-invariant.
We proved that the map $\varphi_{\imath}\colon\mathbb{C}^{2n+1}\rightarrow\Sp(2n+2,\mathbb{C})$
is an injective morphism and $H_{\imath}=\varphi_{\imath}(\mathbb{C}^{2n+1})$
is a subgroup acting on $U_{\imath}$.

Note that $\varphi_{\imath}(\vec{a},\vec{c})\cdot\phi_{\imath}(\vec{b},\vec{d})=\phi_{\imath}(\vec{q},\vec{w})$
where $\vec{q}$ and $\vec{w}$ are the vectors whose coordinates
are in Eq. (\ref{eq:q,w}). It follows that the isomorphism $\phi_{\imath}\circ\varphi_{\imath}^{-1}\colon H_{\imath}\rightarrow U_{\imath}$
is $H_{\imath}$-linear. Hence the action of $H_{\imath}$ on $U_{\imath}$
is free and transitive.

Now, we prove that such a construction works for every
$\imath$.
Suppose $n<\imath\le2n+1$ and let $f$ be the involutive
linear automorphism of $\mathbb{C}^{2n+2}$ such that $f(x_{k})=y_{k}$
for $k=0,\ldots,n$. Let $U_{\imath}$ be the open subset of $\mathbb{P}^{2n+1}$
of vectors with coefficient of $y_{\imath-n-1}$ non-zero, and let
$\phi_{\imath}\colon\mathbb{C}^{2n+1}\rightarrow U_{\imath}$ be the
standard isomorphism. Let us define $\varphi_{\imath}:=f\circ\varphi_{\imath-n-1}\circ f$.
It is clear that such a morphism satisfies the conditions, and $\{U_{\imath}\}_{\imath=0}^{2n+1}$
is the standard open cover.
\end{proof}
\begin{rem}
\label{rem:Matrix_form}We can see that $\varphi_{\imath}(\vec{a},\vec{c})$
is a symplectic linear isomorphism by looking at its matrix description.
With respect to the basis $\{x_{0},\ldots,x_{n},y_{0},\ldots,y_{n}\}$,
let us consider the representation
\[
\varphi_{\imath}(\vec{a},\vec{c})=\begin{pmatrix}A & B\\
C & D
\end{pmatrix},
\]
where $A,B,C,D$ are square matrices of order $(n+1)$. For example,
when $\imath=0$,
\[
A=\begin{pmatrix}1\\
a_{1} & 1\\
\vdots &  & \ddots\\
a_{n} &  &  & 1
\end{pmatrix}C=\begin{pmatrix}c_{0} & c_{1} & \cdots & c_{n}\\
c_{1}\\
\vdots\\
c_{n}
\end{pmatrix}D=\begin{pmatrix}1 & -a_{1} & \cdots & -a_{n}\\
 & 1\\
 &  & \ddots\\
 &  &  & 1
\end{pmatrix}
\]
and $B$ is the zero matrix. It is easy to see that they satisfy Eqs
(\ref{eq:CA}), (\ref{eq:AD}), and (\ref{eq:DB}).
\end{rem}

\section{Contact Stable Maps}

\subsection{Stable trees}

We give an introduction to stable trees as in \cite[Part I]{BM}.
A graph $\tau$ is a quadruple $(F_{\tau},V_{\tau},\partial_{\tau},j_{\tau})$,
where $F_{\tau}$ is a finite set of flags, $V_{\tau}$ is a finite
set of vertices, $\partial_{\tau}\colon F_{\tau}\rightarrow V_{\tau}$
is a map, and $j_{\tau}\colon F_{\tau}\rightarrow F_{\tau}$ is an
involution. The elements of $E_{\tau}:=\{\{f,f'\}\subseteq F_{\tau}\slash f'= j_{\tau}(f) \neq f\}$
are called the edges of $\tau$, and the elements of $L_{\tau}:=\{f\in F_{\tau}\slash j_{\tau}(f)=f\}$
are called the leaves (or tails) of $\tau$. For each vertex $v\in V_{\tau}$,
its valence $n(v)$ is the cardinality of $F_{\tau}(v):=\partial_{\tau}^{-1}(\{v\})$.

An isomorphism of graphs $\varphi\colon\tau\rightarrow\sigma$ is
a pair $(\varphi_{F},\varphi_{V})$, where $\varphi_{F}\colon F_{\tau}\rightarrow F_{\sigma}$
and $\varphi_{V}\colon V_{\tau}\rightarrow V_{\sigma}$ are bijective
maps such that $\varphi_{V}\circ\partial_{\tau}=\partial_{\sigma}\circ\varphi_{F}$
and $\varphi_{F}\circ j_{\tau}=j_{\sigma}\circ\varphi_{F}$. A graph
is connected if for every two vertices $v,v'\in V_{\tau}$, there
exists a sequence of vertices $\{v_{i}\}_{i=0}^{r}$ such that
\begin{itemize}
\item for each $i=1,\ldots r$, there exists an edge $e=\{f,f'\}$ such that $\partial_\tau(f)=v_{i-1}$, $\partial_\tau(f')=v_{i}$, and
\item $v_{0}=v$, $v_{r}=v'$.
\end{itemize}
\begin{defn}
A tree is a connected graph such that $|E_{\tau}|+|L_{\tau}|+|V_{\tau}|=1+|F_{\tau}|$.
\end{defn}

For any pair $(m,d)$ of non negative integers, a $(m,d)$-tree is
the datum $(\tau,l_{\tau},d_{\tau})$ where $\tau$ is a tree, $A_{\tau}$
is a set, $|A_{\tau}|=m$, $l_{\tau}\colon L_{\tau}\rightarrow A_{\tau}$
is a bijective map, and $d_{\tau}\colon V_{\tau}\rightarrow\mathbb{Z}$
is a map such that $\sum_{v\in V_{\tau}}d_{\tau}(v)=d$ and $d_{\tau}(v)\ge0$
for every $v\in V_{\tau}$.
\begin{defn}
A $(m,d)$-tree $(\tau,l_{\tau},d_{\tau})$ is stable if for each
$v\in V_{\tau}$, either $d_{\tau}(v)\neq0$ or $n(v)\ge3$.
\end{defn}

Two stable $(m,d)$-trees $(\tau,l_{\tau},d_{\tau})$ and $(\sigma,l_{\sigma},d_{\sigma})$
are isomorphic if there exist isomorphisms $\varphi\colon\tau\rightarrow\sigma$, $\theta\colon A_\tau\rightarrow A_\sigma$
such that $\theta\circ l_{\tau}=l_{\sigma}\circ\varphi_{F}$ and $d_{\tau}=d_{\sigma}\circ\varphi_{V}$.
By abuse of notation, we may denote a stable $(m,d)$-tree $(\tau,l_{\tau},d_{\tau})$
simply by $\tau$. 
\begin{defn}
The group of automorphisms of the stable $(m,d)$-tree $\tau$ is
denoted by $\mathrm{Aut}(\tau)$.
\end{defn}

The set of all stable $(m,d)$-tree, modulo isomorphism, is denoted
by $\gra$. It is finite by \cite[Proposition 4.3]{GP}. The subset
of those trees such that $d_{\tau}(v)>0$ for every $v\in V_{\tau}$
is denoted by $\grapos$.
\begin{defn}
The unique stable $(m,d)$-tree with only one vertex $v$ is
denoted by $\tau_{v}$.
\end{defn}

Since each edge joints two vertices, in a stable $(m,d)$-tree $\tau$
we have 
\begin{equation}
|F_{\tau}|-2|V_{\tau}|=m-2.\label{eq:relation}
\end{equation}

\subsection{Stable maps}

In this section by $n$ we will denote any odd number.

A stable map is the datum $(C,f,p_{1},\ldots,p_{m})$ where $C$ is
a projective, connected, nodal curve of arithmetic genus $0$, the
markings $p_{1},\ldots,p_{m}$ are distinct nonsingular points of
$C$, and $f\colon C\rightarrow\mathbb{P}^{n}$ is a morphism such
that $f_{*}([C])$ is a $1$-cycle of degree $d$. Moreover, for every
rational component $E\subseteq C$ mapped to a point, $E$ contains
at least three points among marked points and nodal points.

The dual graph $\tau$ of the stable map $(C,f,p_{1},\ldots,p_{m})$
is defined as follows.
\begin{itemize}
\item For each vertex $v\in V_{\tau}$, there is an irreducible component
$C_{v}$ of $C$.
\item For each nodal point of $C$, there is an edge in $\tau$ connecting
the vertices corresponding to those components.
\item For each $p_{i}$, there is a leaf of $\tau$ attached to the vertex
corresponding to the irreducible component containing $p_{i}$.
\item There is a map $d_{\tau}\colon V_{\tau}\rightarrow\mathbb{Z}$ such
that $d_{\tau}(v)$ is the degree of the restriction of $f$ to $C_{v}$.
\end{itemize}
A map $f\colon C\rightarrow\mathbb{P}^{n}$ of genus $0$ with $m$
marked points and degree $d$ is stable if and only if its dual graph
is a stable $(m,d)$-tree \cite[Proposition 4.4]{GP}.

Let $\mathcal{M}_{0,m}(\mathbb{P}^{n},d)$ be the coarse moduli space
of irreducible stable maps of genus $0$ to $\mathbb{P}^{n}$, and
let $\overline{\mathcal{M}}_{0,m}(\mathbb{P}^{n},d)$ be its compactification.
Let us denote by $\ev_{i}$ the evaluation map $\ev_{i}\colon\overline{\mathcal{M}}_{0,m}(\mathbb{P}^{n},d)\rightarrow\mathbb{P}^{n}$
at the marked point $p_{i}$. Given an isomorphism class $\tau$ of
stable $(m,d)$-trees, the locus $\mathcal{M}(\tau)\subset\overline{\mathcal{M}}_{0,m}(\mathbb{P}^{n},d)$
parameterizes maps whose dual graph is isomorphic to $\tau$. It is
locally closed and of codimension $|E_{\tau}|$. For example, the
unique stable $(m,d)$-tree with only one vertex corresponds to $\mathcal{M}_{0,m}(\mathbb{P}^{n},d)$.
There is a stratification of $\overline{\mathcal{M}}_{0,m}(\mathbb{P}^{n},d)$
given by $\mathcal{M}(\tau)$ for all $\tau\in\gra$. Finally, there
is a canonical isomorphism $\mathcal{M}(\tau)\cong\mathcal{M}_{\square}(\tau)\slash\mathrm{Aut}(\tau)$
where $\mathcal{M}_{\square}(\tau)$ is the fibred product
\begin{equation}\label{eq:first_sq}
    \xymatrix{\mathcal{M}_{\square}(\tau)\ar[d]\ar[r] & \prod_{v\in V_{\tau}}\mathcal{M}_{0,F_{\tau}(v)}(\mathbb{P}^{n},d_\tau(v))\ar[d]\\
(\mathbb{P}^{n})^{E_{\tau}\sqcup L_{\tau}}\ar[r] & (\mathbb{P}^{n})^{F_{\tau}}.
}
\end{equation}
We denote by $\overline{\mathcal{M}(\tau)}$ the topological closure
of $\mathcal{M}(\tau)$ in $\overline{\mathcal{M}}_{0,m}(\mathbb{P}^{n},d)$. 
\begin{rem}
\label{rem:closure_M}We may define a scheme $\overline{\mathcal{M}_{\square}(\tau)}$
using the same square of Eq. (\ref{eq:first_sq}), but using $\overline{\mathcal{M}}_{0,F_{\tau}(v)}(\mathbb{P}^{n},d_\tau(v))$
instead of $\mathcal{M}_{0,F_{\tau}(v)}(\mathbb{P}^{n},d_\tau(v))$. Such
a scheme has a natural ramified map to $\overline{\mathcal{M}}_{0,m}(\mathbb{P}^{n},d)$
of degree $\mathrm{Aut}(\tau)$. The image is the closure of $\mathcal{M}(\tau)$.
See \cite{BM} or \cite{MM}.
\end{rem}

Let us now focus on contact curves.
% \begin{defn}
% Let $\{U_i,\alpha_i\}$ be a contact structure on $X$. A contact stable map $(C,f,p_{1},\ldots,p_{m})$
% is a stable map such that $f^{*}\alpha_i=0$ for all $i$.
% \end{defn}
\begin{defn}
Let $(X,L)$ be a contact structure. A contact stable map
is a stable map $(C,f,p_{1},\ldots,p_{m})$ such that for any local section
$s$ of $L$, $f^*s=0$.
\end{defn}

In \cite{Mur} we proved the following properties:
\begin{itemize}
\item The moduli space of contact stable maps $\MC_{m}(\mathbb{P}^{n},d)\subset\overline{\mathcal{M}}_{0,m}(\mathbb{P}^{n},d)$
is the zero locus of a vector bundle $\mathcal{E}$ of rank $2d-1$.
\item If $d>0$, the irreducible components of $\MC_{m}(\mathbb{P}^{n},d)$ are purely
of dimension $d(n-1)+n+m-2$.
\item The Gromov-Witten invariant
\begin{equation}
\int_{\overline{\mathcal{M}}_{0,m}(\mathbb{P}^{n},d)}\ev_{1}^{*}(\Gamma_{1})\cdots\ev_{m}^{*}(\Gamma_{m})\cdot c_{2d-1}(\mathcal{E})=\int_{\MC_{m}(\mathbb{P}^{n},d)}\ev_{1}^{*}(\Gamma_{1})\cdots\ev_{m}^{*}(\Gamma_{m})\label{eq:the_GW_inv}
\end{equation}
is enumerative.
\end{itemize}
Enumerative means that if $\{\Gamma_{i}\}_{i=1}^{m}$ are subvarieties
of $\mathbb{P}^{n}$ in general position such that $\sum_{i=1}^{m}\mathrm{codim}(\Gamma_{i})=\dim\MC_m(\mathbb{P}^{n},d)$,
then (\ref{eq:the_GW_inv}) equals the number of contact curves passing
through all $\Gamma_{i}$.

If the stable map $(C,f,p_{1},\ldots,p_{m})$ contracts
$C$, then it is contact.
% is a stable map such that $f$ contracts
% $C$, then $(C,f,p_{1},\ldots,p_{m})$ is contact.
It follows that $\MC_{m}(\mathbb{P}^{n},0)=\overline{\mathcal{M}}_{0,m}(\mathbb{P}^{n},0)\cong\overline{\mathcal{M}}_{0,m}\times\mathbb{P}^{n}$.
Hence
\begin{equation}
\int_{\MC_{m}(\mathbb{P}^{n},0)}\ev_{1}^{*}(\Gamma_{1})\cdots\ev_{m}^{*}(\Gamma_{m})=\begin{cases}
\Gamma_{1}\cdots\Gamma_{m} & m=3\\
0 & m\neq3.
\end{cases}\label{eq:d=00003D0}
\end{equation}

\begin{rem}
Let $H$ be the generator of $H^2(\mathbb P^n,\mathbb Z)$. It can be defined a ``contact potential'' $\Phi\in\mathbb{Q}\left\llbracket h_{0},\ldots,h_{n}\right\rrbracket $,
like in the usual GW-potential:
\begin{eqnarray*}
I_{d}(m_{0},\cdots,m_{n}) & = & \int_{\MC_{m_{0}+\ldots+m_{n}}(\mathbb{P}^{n},d)}\prod_{i=0}^{n}\ev^{*}(H^{i})^{\cdot m_{i}}\\
\Phi & = & \sum_{m_{0}+\ldots+m_{n}\ge3}\sum_{d=0}^{\infty}I_{d}(m_{0},\cdots,m_{n})\frac{h_{0}^{m_{0}}}{m_{0}!}\cdots\frac{h_{n}^{m_{n}}}{m_{n}!}.
\end{eqnarray*}
For example, in the case $n=3$,
\[
\Phi=h_{0}h_{1}h_{2}+\frac{h_{1}^{3}}{3!}+\frac{h_{0}^{2}h_{3}}{2!}+2\frac{h_{1}^{3}}{3!}+h_{1}h_{2}h_{3}+2\frac{h_{2}h_{3}^{2}}{2}+\cdots.
\]
The WDVV differential equations is the system of PDE:
\[
\sum_{e=0}^{n}\Phi_{ije}\Phi_{kl(n-e)}=\sum_{e=0}^{n}\Phi_{jke}\Phi_{il(n-e)},
\]
where $\Phi_{ijk}:=\frac{\partial^{3}}{\partial h_{i}\partial h_{j}\partial h_{k}}\Phi$.
The usual GW-potential satisfies the WDVV equations \cite{KM}. The
contact potential does not. In order to see this, it is enough to compute all derivatives $\Phi_{ijk}$, and evaluate them at $0$. There are not known differential equations
satisfied by $\Phi$.
\end{rem}

\section{Stratification of the Moduli of Contact Stable Maps}

\begin{defn}
For every $\tau\in\gra$, we denote by $\MC(\tau)$ the space given
by the intersection $\MC_{m}(\mathbb{P}^{n},d)\cap\mathcal{M}(\tau)$.
\end{defn}

We will denote by $\overline{\MC(\tau)}$ the topological closure
of $\MC(\tau)$ in $\MC_{m}(\mathbb{P}^{n},d)$. The restriction of
the evaluation maps $\ev_{i}\colon\overline{\mathcal{M}}_{0,m}(\mathbb{P}^{n},d)\rightarrow\mathbb{P}^{n}$
to $\MC(\tau)$ or $\overline{\MC(\tau)}$ will still be denoted by
$\ev_{i}$. We denote by $\mathcal{S}_{m}(\mathbb{P}^{n},d)^{\circ}$ the open subset $\MC_{m}(\mathbb{P}^{n},d)\cap\mathcal{M}_{m}(\mathbb{P}^{n},d)$.
The following result follows from Proposition \ref{prop:action_Sp}.
See also \cite[Proposition 3.8]{Bagnarol}.
\begin{prop}
\label{prop:ZLT}Let $\tau\in\gra$. The map $\ev_{i}\colon\MC(\tau)\rightarrow\mathbb{P}^{n}$
is Zariski locally trivial.
\end{prop}

\begin{proof}
We start by proving that $\ev_{i}\colon\MC_{m}(\mathbb{P}^{n},d)\rightarrow\mathbb{P}^{n}$
is Zariski locally trivial. As we saw in Proposition \ref{prop:action_Sp},
there exists an open cover of $\mathbb{P}^{n}$ such that each open
set $U$ is isomorphic to a subgroup $H$ of $\Sp(n+1,\mathbb{C})$,
in such a way that $U$ is $H$-invariant and the action of $H$ on
$U$ is free and transitive. This induces an action
% of H on By composition, there exists an action
% of $\Sp(n+1,\mathbb{C})$ on $\MC_{m}(\mathbb{P}^{n},d)$. Since $U$
% is $H$-invariant, it follows that there exists an action
$\beta\colon H\times\ev_{i}^{-1}(U)\rightarrow\ev_{i}^{-1}(U)$ sending $(h,(C,f,p_1,\ldots,p_m))$ to $(C,h\cdot f,p_1,\ldots,p_m)$.
Let $o\in U$ be a point. There exists a map $\gamma\colon H\times\ev_{i}^{-1}(\{o\})\rightarrow\ev_{i}^{-1}(U)$
given by 
%composition 
\[
H\times\ev_{i}^{-1}(\{o\})\rightarrow H\times\ev_{i}^{-1}(U)\rightarrow\ev_{i}^{-1}(U),
\]
where the first map is the embedding and the second is $\beta$.
Let us prove that $\gamma$ is invertible. Since the action of $H$ on
$U$ is free and transitive, there exists a unique isomorphism of
varieties $g\colon U\rightarrow H$ such that $g(x)\cdot x=o$ for
all $x\in U$. Let us consider the map $((g\circ\ev_{i}),\mathrm{id})\colon\ev_{i}^{-1}(U)\rightarrow H\times\ev_{i}^{-1}(U)$.
The composition of that map with $\beta$ has image contained in
$H\times\ev_{i}^{-1}(\{o\})$ and it is the inverse of $f$.

Since $H$ preserves the dual graph of any contact
stable map, it induces an action on any $\MC(\tau)$, in such a way
that $\ev_{i}\colon\MC(\tau)\rightarrow\mathbb{P}^{n}$ is Zariski
locally trivial by restriction.
\end{proof}
\begin{prop}
\label{prop:S(t)}Let $\tau\in\gra$. Then $\mathcal{S}(\tau)\cong\mathcal{S}_{\square}(\tau)\slash\mathrm{Aut}(\tau)$
where $\mathcal{S}_{\square}(\tau)$ is the fibred product
\begin{equation}
\xymatrix{\mathcal{S}_{\square}(\tau)\ar[d]\ar[r] & \prod_{v\in V_{\tau}}\mathcal{S}_{F_{\tau}(v)}(\mathbb{P}^{n},d_\tau(v))^{\circ}\ar[d]\\
(\mathbb{P}^{n})^{E_{\tau}\sqcup L_{\tau}}\ar[r] & (\mathbb{P}^{n})^{F_{\tau}}.
}
\label{eq:S_sq}
\end{equation}
Moreover, $\mathcal{S}(\tau)$ and $\mathcal{S}_{\square}(\tau)$
are purely of the same dimension.
\end{prop}

\begin{proof}
By standard properties of fibred product (e.g., \cite[Proposition 4.16]{GW}),
there is a double fibred product diagram involving $\mathcal{S}_{\square}(\tau)$

\[
\xymatrix{\mathcal{S}_{\square}(\tau)\ar@{^{(}->}[r]\ar[d] & \mathcal{M}_{\square}(\tau)\ar[r]\ar[d] & (\mathbb{P}^{n})^{E_{\tau}\sqcup L_{\tau}}\ar[d]\\
\prod_{v\in V_{\tau}}\mathcal{S}_{F_{\tau}(v)}(\mathbb{P}^{n},d_\tau(v))^{\circ}\ar@{^{(}->}[r] & \prod_{v\in V_{\tau}}\mathcal{M}_{0,F_{\tau}(v)}(\mathbb{P}^{n},d_\tau(v))\ar[r] & (\mathbb{P}^{n})^{F_{\tau}}.
}
\]
That is, $\mathcal{S}_{\square}(\tau)$ is embedded in $\mathcal{M}_{\square}(\tau)$
and it is $\mathrm{Aut}(\tau)$-invariant. This induces an action
of $\mathrm{Aut}(\tau)$ on $\mathcal{S}_{\square}(\tau)$, in particular
$\mathcal{S}_{\square}(\tau)\slash\mathrm{Aut}(\tau)\subseteq\mathcal{M}(\tau)$.
As we said in Remark \ref{rem:contact_local}, a reducible curve is
contact if and only if each one of its components is contact in its
own. That is, if $(C,f,p_{1},\ldots,p_{m})\in\mathcal{M}(\tau)$ is
contact then for every irreducible component $C_{i}\subseteq C$,
the restriction $f_{|C_{i}}\colon C_{i}\rightarrow\mathbb{P}^{n}$
is contact too. Thus both $\mathcal{S}_{\square}(\tau)\slash\mathrm{Aut}(\tau)\subseteq\mathcal{S}(\tau)$
and $\mathcal{S}(\tau)\subseteq\mathcal{S}_{\square}(\tau)\slash\mathrm{Aut}(\tau)$
are true.

The fact that $\mathcal{S}_{\square}(\tau)$ is purely dimensional
follows from Proposition \ref{prop:ZLT}. Indeed, each $\mathcal{S}_{F_{\tau}(v)}(\mathbb{P}^{n},d_\tau(v))^{\circ}$
is purely dimensional, so each fiber $F$ of the locally trivial map
$\mathcal{S}_{F_{\tau}(v)}(\mathbb{P}^{n},d_\tau(v))^{\circ}\rightarrow\mathbb{P}^{n}$
is also. It follows that $\mathcal{S}_{\square}$ is locally isomorphic
to $F\times U$ where $U$ is an open subset of $(\mathbb{P}^{n})^{E_{\tau}\sqcup L_{\tau}}$.
Since $\mathrm{Aut}(\tau)$ is finite, $\mathcal{S}(\tau)$ is also
purely dimensional and $\dim\MC(\tau)=\dim\mathcal{S}_{\square}(\tau)$.
\end{proof}
Even if we use the fact that $\mathcal{S}_{m}(\mathbb{P}^{n},d)^{\circ}$
is purely dimensional, we do not say that $\mathcal{S}_{F_{\tau}(v)}(\mathbb{P}^{n},d_\tau(v))^{\circ}$
is irreducible. As already noted in the proof of \cite[Lemma 3.7]{Mur},
$\mathcal{S}_{m}(\mathbb{P}^{n},d)^{\circ}$ is a quotient $\mathcal{M}_{d}\slash\mathrm{Aut}(\mathbb{P}^{1})$
where $\mathcal{M}_{d}$ is the moduli space of harmonic maps $\mathbb{C}\mathbb{P}^{1}\rightarrow\mathbb{C}\mathbb{P}^{3}$
we cited in Introduction. The manifold $\mathcal{M}_{d}$ need to
be irreducible \cite{Loo,KL}. In the same Lemma, it is claimed that
% the dimension of $\MC_1(\mathbb P^n,d_1)\times_{\mathbb P^n} \MC_1(\mathbb P^n,d_2)$ is 
$$\dim(\MC_1(\mathbb P^n,d_1)\times_{\mathbb P^n} \MC_1(\mathbb P^n,d_2))=\dim \MC_1(\mathbb P^n,d_1)+\dim \MC_1(\mathbb P^n,d_2)-n.$$
This follows from the fact the action of $\Sp(n+1,\mathbb{C})$ on $\mathbb P^n$ is transitive and preserves the fibers of $\ev_1\colon\MC_1(\mathbb P^n,d)\rightarrow\mathbb P^n$. So $\ev_1$ is flat.

We are now ready to state the first
important result of the paper.
\begin{thm}
\label{thm:codim}Let $n$ be an odd number and $\MC_{m}(\mathbb{P}^{n},d)$
be the moduli space of contact stable maps. 
\begin{enumerate}
\item Every subspace $\MC(\tau)$ is locally closed in $\MC_{m}(\mathbb{P}^{n},d)$.
\item There exists a stratification of $\MC_{m}(\mathbb{P}^{n},d)$ given
by the union of all $\MC(\tau)$.% for $\tau\in\gra$.
\item For every $\tau\in\gra$, the spaces $\MC(\tau)$ are of pure codimension
equal to $|V_{\tau}^{0}|$, where $V_{\tau}^{0}=\{v\in V_{\tau}\slash d_\tau(v)=0\}$.
\end{enumerate}
In particular there exists a decomposition of $\MC_{m}(\mathbb{P}^{n},d)$
as union of components $\overline{\MC(\tau)}$ for $\tau\in\grapos$,
so that
\begin{equation}
\int_{\MC_{m}(\mathbb{P}^{n},d)}\ev_{1}^{*}(\Gamma_{1})\cdots\ev_{m}^{*}(\Gamma_{m})=\sum_{\tau\in\grapos}\int_{\overline{\MC(\tau)}}\ev_{1}^{*}(\Gamma_{1})\cdots\ev_{m}^{*}(\Gamma_{m}).\label{eq:Sum_plus}
\end{equation}
\end{thm}

\begin{proof}
The first two points follow from analogous properties of $\overline{\mathcal{M}}_{0,m}(\mathbb{P}^{n},d)$,
so let us prove Point $(3)$. By Proposition \ref{prop:S(t)} we have
\[
\dim\mathcal{S}(\tau)=n(|E_{\tau}|+|L_{\tau}|-|F_{\tau}|)+\sum_{v\in V_{\tau}}\dim\mathcal{S}_{F_{\tau}(v)}(\mathbb{P}^{n},d_\tau(v))^{\circ}.
\]
Recall that if $d>0$ (resp., $d=0$) then $\dim\mathcal{S}_{m}(\mathbb{P}^{n},d)^{\circ}=d(n-1)+n+m-2$
(resp., $n+m-3$). Thus the dimension of $\mathcal{S}(\tau)$ is, using Eq. (\ref{eq:relation}),

\begin{eqnarray*}
&   & n(1-|V_{\tau}|)+\sum_{v\in V_{\tau}^{0}}\mathcal{S}_{F_{\tau}(v)}(\mathbb{P}^{n},d_\tau(v))^{\circ}+\sum_{v\in V_{\tau}\setminus V_{\tau}^{0}}\mathcal{S}_{F_{\tau}(v)}(\mathbb{P}^{n},d_\tau(v))^{\circ}  \\
&    = & n(1-|V_{\tau}|)+n|V_{\tau}|+|F_{\tau}|-3|V_{\tau}^{0}|+d(n-1)-2(|V_{\tau}|-|V_{\tau}^{0}|)  \\
&   = & n+|F_{\tau}|-2|V_{\tau}|-|V_{\tau}^{0}|+d(n-1)  \\
& = &  d(n-1)+n+m-2-|V_{\tau}^{0}|  \\
&  =  & \dim\mathcal{S}_{m}(\mathbb{P}^{n},d)-|V_{\tau}^{0}|. 
\end{eqnarray*}

% \[
% \begin{array}{cl}
%  & n(1-|V_{\tau}|)+\sum_{v\in V_{\tau}^{0}}\mathcal{S}_{F_{\tau}(v)}(\mathbb{P}^{n},d(v))^{\circ}+\sum_{v\in V_{\tau}\setminus V_{\tau}^{0}}\mathcal{S}_{F_{\tau}(v)}(\mathbb{P}^{n},d(v))^{\circ}\\
% = & n(1-|V_{\tau}|)+n|V_{\tau}|+|F_{\tau}|-3|V_{\tau}^{0}|+d(n-1)-2(|V_{\tau}|-|V_{\tau}^{0}|)\\
% = & n+|F_{\tau}|-2|V_{\tau}|-|V_{\tau}^{0}|+d(n-1)\\
% = & d(n-1)+n+m-2-|V_{\tau}^{0}|\\
% = & \dim\mathcal{S}_{m}(\mathbb{P}^{n},d)-|V_{\tau}^{0}|.
% \end{array}
% \]
In particular for those $\tau\in\grapos$
we have $\dim\MC(\tau)=\dim\MC_{m}(\mathbb{P}^{n},d)$. Since $\MC_{m}(\mathbb{P}^{n},d)$
is purely of the expected dimension, it can be decomposed as union
of $\overline{\MC(\tau)}$ for $\tau\in\grapos$. Equation (\ref{eq:Sum_plus})
follows from linearity of the intersection product \cite[Example 1.8.1]{MR1644323}.
\end{proof}

In the next section, we will see how to use the theorem to compute all enumerative invariants of contact curves. 
\section{Applications}

Let us briefly introduce the decomposition of stable trees. Given
a stable $(m,d)$-tree $(\tau,l_{\tau},d_{\tau})$, for every edge
$\{f,f'\}=e\in E_{\tau}$ there exist two unique stable trees $(\sigma,l_{\sigma},d_{\sigma})$
and $(\sigma',l_{\sigma'},d_{\sigma'})$ with the following properties:
\begin{itemize}
\item there is a partition $F_{\tau}=F_{\sigma}\sqcup F_{\sigma'}$ such
that $f\in F_{\sigma}$ and $f'\in F_{\sigma'}$,
\item for every $x\in F_{\tau}\backslash\{f,f'\}$ such that $x\in F_{\sigma}$
(resp., $F_{\sigma'}$), $j_{\sigma}(x)$ (resp., $j_{\sigma'}(x)$)
is equal to $j_{\tau}(x)$,
\item $f$ and $f'$ are fixed points of, respectively, $j_{\sigma}$ and
$j_{\sigma'}$,
\item $V_{\sigma}=\partial_{\tau}(F_{\sigma})$ and $V_{\sigma'}=\partial_{\tau}(F_{\sigma'})$, and,
\item $d_{\sigma}$ (resp., $d_{\sigma'}$) is the restriction of $d_{\tau}$
to $V_{\sigma}$ (resp., $V_{\sigma'}$).
\end{itemize}
It is immediate to deduce that $E_{\tau}\backslash\{e\}=E_{\sigma}\sqcup E_{\sigma'}$
and $L_{\tau}\sqcup\{f,f'\}=L_{\sigma}\sqcup L_{\sigma'}$. We will
use the following notation: $E_{\tau}^{*}=E_{\tau}\backslash\{e\}$,
$L_{\sigma}^{*}=L_{\sigma}\backslash\{f\}$ and $L_{\sigma'}^{*}=L_{\sigma'}\backslash\{f'\}$.
The maps $l_{\sigma}$ and $l_{\sigma'}$ are defined in the obvious
way.

We will denote by $\tau=\sigma\times_{e}\sigma'$ the decomposition
of $\tau$ with respect to the edge $e$. Note that every vertex $v$
of $\sigma$ or $\sigma'$ has the same valence as a vertex of $\tau$,
so $\sigma$ and $\sigma'$ are always stable.
\begin{cor}
\label{cor:S_closed}Let $\tau\in\gra$ such that $\tau=\sigma\times_{e}\sigma'$
for an edge $e=\{f,f'\}$. Then 
\begin{enumerate}
\item There exists a projective variety  $\overline{\MC_{\square}(\tau)}$ and a map $\overline{\MC_{\square}(\tau)}\rightarrow\overline{\MC(\tau)}$
extending $\MC_{\square}(\tau)\rightarrow\MC(\tau)$.
\item There exist isomorphisms
$$
\begin{array}{cc}
\mathcal{S}_{\square}(\tau)\cong\mathcal{S}_{\square}(\sigma)\times_{\mathbb{P}^{n}}\mathcal{S}_{\square}(\sigma'), & \overline{\MC_{\square}(\tau)}\cong\overline{\MC_{\square}(\sigma)}\times_{\mathbb{P}^{n}}\overline{\MC_{\square}(\sigma')}.
\end{array}
$$
\end{enumerate}
\end{cor}

\begin{proof}
Let us define $\overline{\MC_{\square}(\tau)}$ as the fiber product
\[
\xymatrix{\overline{\MC_{\square}(\tau)}\ar[d]\ar[r] & \prod_{v\in V_{\tau}}\overline{\mathcal{S}_{F_{\tau}(v)}(\mathbb{P}^{n},d_\tau(v))^{\circ}}\ar[d]\\
(\mathbb{P}^{n})^{E_{\tau}\sqcup L_{\tau}}\ar[r] & (\mathbb{P}^{n})^{F_{\tau}}.
}
\]
Point $(1)$ follows from Remark \ref{rem:closure_M}: The image of
the natural map $\MC_{\square}(\tau)\hookrightarrow\overline{\MC_{\square}(\tau)}$
is dense in every irreducible component of $\overline{\MC_{\square}(\tau)}$,
and it factorizes the inclusion $\MC_{\square}(\tau)\hookrightarrow\overline{\mathcal{M}_{\square}(\tau)}$.
It follows that the restriction of $\overline{\mathcal{M}_{\square}(\tau)}\rightarrow\overline{\mathcal{M}}_{0,m}(\mathbb{P}^{n},d)$
to $\overline{\MC_{\square}(\tau)}$ maps to $\overline{\MC(\tau)}$.

Let us prove Point $(2)$. The isomorphism $V_{\tau}\cong V_{\sigma}\sqcup V_{\sigma'}$
preserves the valence of each vertex $v\in V_{\tau}$. So there exists
an isomorphism
\[
X:=\prod_{v\in V_{\tau}}\mathcal{S}_{F_{\tau}(v)}(\mathbb{P}^{n},d_{\tau}(v))^{\circ}\cong\prod_{v\in V_{\sigma}}\mathcal{S}_{F_{\sigma}(v)}(\mathbb{P}^{n},d_{\sigma}(v))^{\circ}\times\prod_{v\in V_{\sigma'}}\mathcal{S}_{F_{\sigma'}(v)}(\mathbb{P}^{n},d_{\sigma'}(v))^{\circ}.
\]
Using the equality $F_{\tau}=F_{\sigma}\sqcup F_{\sigma'}$, it follows
that we have a diagram
\begin{equation}
\xymatrix{\mathcal{S}_{\square}(\tau)\ar[r]\ar[d] & \mathcal{S}_{\square}(\sigma)\times\mathcal{S}_{\square}(\sigma')\ar[r]\ar[d] & X\ar[d]\\
(\mathbb{P}^{n})^{e}\times(\mathbb{P}^{n})^{E_{\tau}^{*}\sqcup L_{\tau}}\ar[r] & (\mathbb{P}^{n})^{f}\times(\mathbb{P}^{n})^{f'}\times(\mathbb{P}^{n})^{E_{\sigma}\sqcup L_{\sigma}^{*}\sqcup E_{\sigma'}\sqcup L_{\sigma'}^{*}}\ar[r] & (\mathbb{P}^{n})^{F_{\tau}}.
}
\label{eq:S_prod}
\end{equation}

The right-hand square is a fibred product diagram, it is obtained
by (\ref{eq:S_sq}) applied to $\sigma$ and $\sigma'$. By applying
again \cite[Proposition 4.16]{GW}, we deduce that also the left-hand
square is a fibred product diagram. The immersion $(\mathbb{P}^{n})^{e}\rightarrow(\mathbb{P}^{n})^{f}\times(\mathbb{P}^{n})^{f'}$
is the diagonal, hence $\mathcal{S}_{\square}(\tau)\cong\mathcal{S}_{\square}(\sigma)\times_{\mathbb{P}^{n}}\mathcal{S}_{\square}(\sigma')$.
The same proof applies equally well in $\overline{\MC_{\square}(\tau)}\cong\overline{\MC_{\square}(\sigma)}\times_{\mathbb{P}^{n}}\overline{\MC_{\square}(\sigma')}$.
\end{proof}
The next corollary allows us to compute intersections in $\overline{\MC(\tau)}$ recursively.
\begin{cor}
\label{cor:int_prod}Let $\tau\in\gra$ such that $\tau=\sigma\times_{e}\sigma'$
for an edge $e=\{f,f'\}$. Then 
\[
\int_{\overline{\MC(\tau)}}\ev_{1}^{*}(\Gamma_{1})\cdots\ev_{m}^{*}(\Gamma_{m})
\]
is equal to

\begin{equation}
\frac{|\mathrm{Aut}(\sigma)||\mathrm{Aut}(\sigma')|}{|\mathrm{Aut}(\tau)|}\sum_{j=0}^{n}\int_{\overline{\MC(\sigma)}}\prod_{i\in L_{\sigma}}\ev_{i}^{*}(\Gamma_{i})\cdot\ev_{f}^{*}(H^{n-j})\,\int_{\overline{\MC(\sigma')}}\prod_{i\in L_{\sigma'}}\ev_{i}^{*}(\Gamma_{i})\cdot\ev_{f'}^{*}(H^{j}).\label{eq:int_prod}
\end{equation}
\end{cor}

\begin{proof}
In Corollary \ref{cor:S_closed} we proved that the map $\overline{\MC_{\square}(\tau)}\rightarrow\overline{\MC(\tau)}$
is generically of degree $|\mathrm{Aut}(\tau)|$, so by projection
formula (see \cite[1.4]{MR1644323})
\[
\int_{\overline{\MC(\tau)}}\ev_{1}^{*}(\Gamma_{1})\cdots\ev_{m}^{*}(\Gamma_{m})=\frac{1}{|\mathrm{Aut}(\tau)|}\int_{\overline{\MC_{\square}(\tau)}}\ev_{1}^{*}(\Gamma_{1})\cdots\ev_{m}^{*}(\Gamma_{m}).
\]
%Using again the corollary, %$\overline{\MC_{\square}(\tau)}\cong\overline{\MC_{\square}(\sigma)}\times_{\mathbb{P}^{n}}\overline{\MC_{\square}(\sigma')}$
Let us denote again by $H$ the generator of $H^2(\mathbb P^n,\mathbb Z)$. The class of the diagonal of $\mathbb P^n\times \mathbb P^n$ is $\sum_{j=0}^n H^{n-j}\boxtimes H^j $, 
so using again the corollary:

\begin{eqnarray*}
& &  \int_{\overline{\MC_{\square}(\sigma)}\times_{\mathbb{P}^{n}}\overline{\MC_{\square}(\sigma')}}\ev_{1}^{*}(\Gamma_{1})\cdots\ev_{m}^{*}(\Gamma_{m})  \\
 & = & \int_{\overline{\MC_{\square}(\sigma)}\times\overline{\MC_{\square}(\sigma')}}\ev_{1}^{*}(\Gamma_{1})\cdots\ev_{m}^{*}(\Gamma_{m})\cdot\sum_{j=0}^{n}\ev_{f}^{*}(H^{n-j})\cdot\ev_{f'}^{*}(H^{j}) \\
 & = &  \sum_{j=0}^{n}\int_{\overline{\MC_{\square}(\sigma)}}\prod_{i\in L_{\sigma}}\ev_{i}^{*}(\Gamma_{i})\cdot\ev_{f}^{*}(H^{n-j}) \int_{\overline{\MC_{\square}(\sigma')}}\prod_{i\in L_{\sigma'}}\ev_{i}^{*}(\Gamma_{i})\cdot\ev_{f'}^{*}(H^{j}) \\
 & = &  \frac{|\mathrm{Aut}(\sigma)|}{|\mathrm{Aut}(\sigma')|^{-1}}
  \sum_{j=0}^{n}\int_{\overline{\MC(\sigma)}}\prod_{i\in L_{\sigma}}\ev_{i}^{*}(\Gamma_{i})\cdot\ev_{f}^{*}(H^{n-j}) \int_{\overline{\MC(\sigma')}}\prod_{i\in L_{\sigma'}}\ev_{i}^{*}(\Gamma_{i})\cdot\ev_{f'}^{*}(H^{j}).
\end{eqnarray*}

% \begin{equation*}
%     \begin{array}{cl}
%  & \int_{\overline{\MC_{\square}(\sigma)}\times_{\mathbb{P}^{n}}\overline{\MC_{\square}(\sigma')}}\ev_{1}^{*}(\Gamma_{1})\cdots\ev_{m}^{*}(\Gamma_{m})\\
% = & \int_{\overline{\MC_{\square}(\sigma)}\times\overline{\MC_{\square}(\sigma')}}\ev_{1}^{*}(\Gamma_{1})\cdots\ev_{m}^{*}(\Gamma_{m})\cdot\sum_{j=0}^{n}\ev_{f}^{*}(H^{n-j})\ev_{f'}^{*}(H^{j})\\
% = & \sum_{j=0}^{n}\int_{\overline{\MC_{\square}(\sigma)}}\prod_{i\in L_{\sigma}}\ev_{i}^{*}(\Gamma_{i})\ev_{f}^{*}(H^{n-j})\cdot \int_{\overline{\MC_{\square}(\sigma')}}\prod_{i\in L_{\sigma'}}\ev_{i}^{*}(\Gamma_{i})\ev_{f'}^{*}(H^{j})
% \end{array}
% \end{equation*}
% So, finally
% \begin{multline*}
%     |\mathrm{Aut}(\sigma)||\mathrm{Aut}(\sigma')|\sum_{j=0}^{n}\int_{\overline{\MC(\sigma)}}\prod_{i\in L_{\sigma}}\ev_{i}^{*}(\Gamma_{i})\ev_{f}^{*}(H^{n-j})\\ \cdot\int_{\overline{\MC(\sigma')}}\prod_{i\in L_{\sigma'}}\ev_{i}^{*}(\Gamma_{i})\ev_{f'}^{*}(H^{j}).
% \end{multline*}

The last equation implies (\ref{eq:int_prod}).
\end{proof}
% In particular, for every $\tau\in\grapos$ with at least two vertices,
% every integral like (\ref{eq:int_S(t)}) can be fully determined recursively
% by the same integrals involving only $\tau_{v}\in\Gamma(m',d')$,
% where $m'\in\mathbb{Z}$ and $d'<d$.
\begin{cor}
\label{cor:shape}Let $\tau\in\gra$ and let $\Gamma_{1},...,\Gamma_{m}$
be subvarieties of $\mathbb{P}^{n}$ in general position such that
\begin{equation}
\sum_{i=1}^{m}\mathrm{codim}(\Gamma_{i})=\dim\MC(\tau).\label{eq:dim_cond_tau}
\end{equation}
Then the number of contact stable maps in $\mathbb{P}^{n}$ of degree
$d$ whose images meet all $\Gamma_{1},...,\Gamma_{m}$, and with
dual graph $\tau$, is

\begin{equation}
\int_{\overline{\MC(\tau)}}\ev_{1}^{*}(\Gamma_{1})\cdots\ev_{m}^{*}(\Gamma_{m}).\label{eq:int_S(t)}
\end{equation}
In particular, if $\sum_{i=1}^{m}\mathrm{codim}(\Gamma_{i})=d(n-1)+n+m-2$,
the number of irreducible contact curves in $\mathbb{P}^{n}$ of degree
$d$ meeting all $\Gamma_{1},...,\Gamma_{m}$ is 
\begin{equation}
\int_{\MC_{m}(\mathbb{P}^{n},d)}\ev_{1}^{*}(\Gamma_{1})\cdots\ev_{m}^{*}(\Gamma_{m})-\sum_{\tau\in\grapos\backslash\{\tau_{v}\}}\int_{\overline{\MC(\tau)}}\ev_{1}^{*}(\Gamma_{1})\cdots\ev_{m}^{*}(\Gamma_{m}).\label{eq:sub}
\end{equation}
\end{cor}

\begin{proof}
This proof follows the same line of \cite[Theorem 3.8]{Mur}, let
us give a sketch. Let $G$ be the group of invertible complex matrices
of order $n+1$, and let $X=\mathbb{P}^{n}$. Consider the diagram
\[
\xymatrix{ & \overline{\MC(\tau)}\ar[d]^{(\ev_{1}\times\cdots\times\ev_{m})}\\
\Gamma_{1}\times\cdots\times\Gamma_{m}=:\Gamma\ar[r] & X^{\times m}.
}
\]
By applying the Kleiman-Bertini Theorem \cite{Kle}, we
deduce that for a general $\sigma\in G^{\times m}$, $\Gamma^{\sigma}\times_{X^{\times m}}\overline{\MC(\tau)}$
is either empty, or of dimension $0$. We assume, without loss of
generality, that $\Gamma^{\sigma}$ is smooth by applying again Kleiman-Bertini.
In the same way we may prove\footnote{Details can be found in \cite[Lemma p. 5]{Tho98} or \cite[Lemma 14]{FP}.}
that the intersection 
\[
\overline{\MC(\tau)}\cap\ev_{1}^{-1}(g_{1}\Gamma_{1})\cap\cdots\cap\ev_{m}^{-1}(g_{m}\Gamma_{m})
\]
is supported in the intersection between the automorphisms-free part
of $\overline{\mathcal{M}}_{0,m}(X,d)$ and the smooth part of $\overline{\MC(\tau)}$.
So, when $\Gamma^{\sigma}\times_{X^{\times m}}\overline{\MC(\tau)}$
is not empty, it is reduced of dimension $0$. Points of that intersection
represent maps whose images meet all $\Gamma_{1},...,\Gamma_{m}$,
and it is equal to (\ref{eq:int_S(t)}).

Finally, $\dim\MC(\tau)=d(n-1)+n+m-2$ implies that $\tau$ has no
contracted components by Theorem \ref{thm:codim}. So if we are interested
in a unique non-contracted component, it must be $\tau=\tau_{v}$.
Equation (\ref{eq:sub}) follows from (\ref{eq:Sum_plus}).
\end{proof}

\begin{rem}\label{rem:all_comp}
The last results imply that for every $\tau\in\gra$, we can compute every integral like in (\ref{eq:int_S(t)}). Indeed, if $d=1$ then $\overline{\MC(\tau)}=\MC_m(\mathbb P^n,1)$
% \begin{equation*}
%     \int_{\overline{\MC(\tau)}}\ev_{1}^{*}(\Gamma_{1})\cdots\ev_{m}^{*}(\Gamma_{m})=\int_{\MC_m(\mathbb P^n,1)}\ev_{1}^{*}(\Gamma_{1})\cdots\ev_{m}^{*}(\Gamma_{m}).
% \end{equation*}
and we can use the package \verb!AtiyahBott! of \cite{MurSch}. If $d>1$, then Eqs. (\ref{eq:int_prod}) and (\ref{eq:sub}) allow us to reduce the computation to intersections in $\MC_m(\mathbb P^n,d)$ and $\MC(\tau')$, for some  $\tau'\in\Gamma(m',d')$ where $m'\in\mathbb{Z}$ and $d'<d$. The result follows by using again the package and recursion on $d$.
\end{rem}
Let us see an example of these applications.
\begin{example}
\label{exa:conic}Suppose we want to compute the number of \textit{irreducible}
contact curves of degree $2$ in $\mathbb{P}^{3}$ passing through
two points and a line in general position. Let $H$ be the class of
a hyperplane in $\mathbb{P}^{3}$. The number of contact curves is
given by
\[
\int_{\MC_{3}(\mathbb{P}^{3},2)}\ev_{1}^{*}(H^{2})\cdot\ev_{2}^{*}(H^{3})\cdot\ev_{3}^{*}(H^{3})=2.
\]
The set $\Gamma(3,2)^{+}$ consists of the following stable trees:

{ \centering \begin {tikzpicture}[auto ,node distance =1.3 cm,on grid , thick , state/.style ={ text=black }]  \node[state] (A1){$2$}; \node[] at (0,0.5) {$\tau_1$};  \node[] at (0,-0.5) {$\{1,2,3\}$};  %\node[state] (B1) [right =of A1] {$q_1$};  %\path (A1) edge node{$1$} (B1);
\node[state] (A2)[right =of A1]{$1$}; \node[] at (1.3,-0.5) {$\{1,2\}$}; \node[] at (2.6,-0.5) {$\{3\}$};; \node[] at (2,0.5) {$\tau_2$}; \node[state] (B2) [right =of A2] {$1$}; \path (A2) edge node{} (B2);
\node[state] (A3) [right =of B2] {$1$}; \node[] at (3.9,-0.5) {$\{1,3\}$}; \node[] at (5.2,-0.5) {$\{2\}$}; \node[] at (4.6,0.5) {$\tau_3$}; \node[state] (B3) [right =of A3] {$1$};  \path (A3) edge node{} (B3);
\node[state] (A4)[right =of B3]{$1$}; \node[] at (6.5,-0.5) {$\{2,3\}$}; \node[] at (7.8,-0.5) {$\{1\}$}; \node[] at (7.2,0.5) {$\tau_4$}; \node[state] (B4) [right =of A4] {$1$};  \path (A4) edge node{} (B4);
\node[state] (A5) [right =of B4] {$1$}; \node[] at (9.1,-0.5) {$\{1,2,3\}$}; \node[] at (10.4,-0.5) {}; \node[] at (9.8,0.5) {$\tau_5$}; \node[state] (B5) [right =of A5] {$1$};  \path (A5) edge node{} (B5);
\end{tikzpicture} \par}
The vertices are labeled by their degree, under any vertex $v$ we
put its leaves. We want to compute $\int_{\overline{\MC(\tau_{1})}}\ev_{1}^{*}(H^{2})\cdot\ev_{2}^{*}(H^{3})\cdot\ev_{3}^{*}(H^{3})$.
Note that $\tau_{2}=\sigma\times_{e}\sigma'$ where

{ \centering \begin {tikzpicture}[auto ,node distance =3 cm,on grid , thick , state/.style ={ text=black }]  \node[state] (A1){$1$}; \node[] at (0,0.5) {$\sigma$};  \node[] at (0,-0.5) {$\{1,2,f\}$};  \node[state] (A2)[right =of A1]{$1$}; \node[] at (3,-0.5) {$\{3,f'\}$}; \node[] at (3,0.5) {$\sigma'$}; \end{tikzpicture} \par}

Using Corollary \ref{cor:int_prod} and Remark \ref{rem:all_comp}, we get
\begin{eqnarray*}
& & \int_{\overline{\MC(\tau_{2})}}\ev_{1}^{*}(H^{2})\cdot\ev_{2}^{*}(H^{3})\cdot\ev_{3}^{*}(H^{3}) \\
& = & \sum_{j=0}^{3}\int_{\overline{\MC(\sigma)}}\ev_{1}^{*}(H^{2})\cdot\ev_{2}^{*}(H^{3})\cdot\ev_{f}^{*}(H^{3-j})\int_{\overline{\MC(\sigma')}}\ev_{3}^{*}(H^{3})\cdot\ev_{f'}^{*}(H^{j})\\
& = & \int_{\overline{\MC(\sigma)}}\ev_{1}^{*}(H^{2})\cdot\ev_{2}^{*}(H^{3})\cdot\ev_{f}^{*}(H^{1})
\int_{\overline{\MC(\sigma')}}\ev_{3}^{*}(H^{3})\cdot\ev_{f'}^{*}(H^{2})\\
& = & 1\cdot1=1.
\end{eqnarray*}

The computation of $\int_{\overline{\MC(\tau_{3})}}\ev_{1}^{*}(H^{2})\cdot\ev_{2}^{*}(H^{3})\cdot\ev_{3}^{*}(H^{3})=1$
is similar. Finally, one can easily figure out that $\int_{\overline{\MC(\tau)}}\ev_{1}^{*}(H^{2})\cdot\ev_{2}^{*}(H^{3})\cdot\ev_{3}^{*}(H^{3})=0$
for $\tau\in\{\tau_{4},\tau_{5}\}$. It follows that
\[
\int_{\overline{\MC(\tau_{1})}}\ev_{1}^{*}(H^{3})\cdot\ev_{2}^{*}(H^{3})\cdot\ev_{3}^{*}(H^{2})=2-1-1-0-0=0.
\]
So, there are no irreducible contact conics through two points and
a line in general position. This result is completely expected since
there are no irreducible contact conics in $\mathbb{P}^{3}$ (\cite[Proposition 17.1]{LM07}). Anyway,
this shows that our method works.
\end{example}

\begin{rem}
\label{rem:no_Aut}If the value of (\ref{cor:int_prod}) is non-zero,
then $|\mathrm{Aut}(\tau)|=1$. Indeed, if  $\sigma'=\tau_{v}$
where $v$ is an extremal vertex with no leaves, then Eq. (\ref{cor:int_prod})
becomes 
%has only a vertex $v$ with only one edge $e$ and no leaves. Eq.(\ref{cor:int_prod}) becomes 
\[
\frac{|\mathrm{Aut}(\sigma)|}{|\mathrm{Aut}(\tau)|}\sum_{j=0}^{n}\int_{\overline{\MC(\sigma)}}\prod_{i\in L_{\sigma}}\ev_{i}^{*}(\Gamma_{i})\cdot\ev_{f}^{*}(H^{n-j})\,\int_{\overline{\MC(\tau_{v})}}\ev_{f'}^{*}(H^{j}).
\]
The value of $\int_{\overline{\MC(\tau_{v})}}\ev_{f'}^{*}(H^{j})$
is zero for dimensional reasons, meaning that also (\ref{cor:int_prod})
is zero. This implies that every vertex of $\tau$ with only one edge
must have leaves. Since any automorphism preserves the leaves, all
vertices with only one edge must be fixed by any automorphism, hence
the identity is the only possible automorphism. Moreover, the sum
(\ref{cor:int_prod}) is non zero for at most one value of $j$, which
is $j_{0}=\dim\MC(\tau_{v})-\sum_{i\in L_{\tau_{v}}}\mathrm{codim}(\Gamma_{i})$,
if $0\le j_{0}\le n$.
\end{rem}
The following result is \cite[Proposition 2.3(i)]{levcovitz2011symplectic}. We provide another proof. 
% Another example of how our method proves already known results is
% the following.
\begin{prop}
\label{prop:LV}The number of contact plane curves of degree $d$
meeting $d+3$ general lines in $\mathbb{P}^{3}$ is 
\[
\frac{d^{2}}{6}(d+3)(d+2)(d+1)(d-1)=20d\binom{d+3}{5}.
\]
\end{prop}

\begin{proof}
We give a sketch of the proof, showing that it can be proved using
graphs. A contact plane curve is linearly degenerate in $\mathbb{P}^{3}$,
then it is cone \cite[Proposition 17.1]{LM07}. On the other hand,
cones in $\mathbb{P}^{3}$ are contact if they are linearly degenerate\footnote{Otherwise, the symplectic form defining the contact structure in $\mathbb{C}^{4}$
would be degenerate.}. The cases $d=1,2$ are trivial, so let us focus on the case $d\ge3$.
We are looking for curves given by the union of $d$ contact lines,
all of them passing through a fixed point. Let $X$ be one of those
contact curves. Any stable map $(C,f,p_{1},\ldots,p_{m})$ with dual
graph $\tau$ whose image is $X$ must have $m=d+3$ marked points
(since we are imposing $d+3$ Schubert conditions). By condition (\ref{eq:dim_cond_tau}),
the number of contracted vertices of degree $0$ of $\tau$ is $d-2$.
The inverse image under $f$ of the singular point of $X$ must be
connected. Hence, the degree $0$ vertices form a connected subtree.
Since we are not imposing any condition on the singular point, no
leaf is attached to a degree 0 vertex. So all leaves lie on $d$
non-contracted components of degree $1$.

The dual graph $\tau$ with these properties is not univoquely determinated.
An example is $\tau_{1}$ in Figure \ref{fig:tau1} (we use the same
notation of Example \ref{exa:conic}).

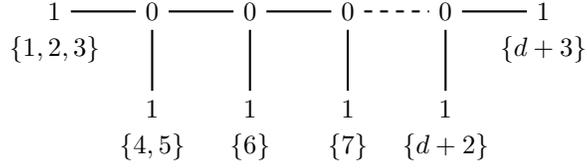
\begin{figure}[h]
{ \centering \begin {tikzpicture}[auto ,node distance =1.3 cm,on grid , thick , state/.style ={ text=black }]  \node[state] (A1){$1$}; \node[] at (0,-0.5) {$\{1,2,3\}$}; 
\node[state] (A2)[right =of A1]{$0$}; \path (A1) edge node{} (A2); \node[] at (1.3,-1.8) {$\{4,5\}$}; \node[state] (C2)[below =of A2]{$1$}; \path (A2) edge node{} (C2); \node[state] (A3) [right =of A2] {$0$}; \path (A2) edge node{} (A3);
\node[state] (C2) [below =of A3] {$1$}; \path (A3) edge node{} (C2); \node[] at (2.6,-1.8) {$\{6\}$};
\node[state] (A4) [right =of A3] {$0$}; \path (A4) edge node{} (A3);
\node[state] (T) [below =of A4] {$1$}; \path (A4) edge node{} (T); \node[] at (3.9,-1.8) {$\{7\}$};
\node[state] (A5) [right =of A4] {$0$}; \path (A4) [dashed] edge node{} (A5); %\node[dashed] at (4.5,0) {};
\node[state] (C4) [below =of A5] {$1$}; \path (A5) edge node{} (C4); \node[] at (5.2,-1.8) {$\{d+2\}$};
\node[state] (A6) [right =of A5] {$1$}; \path (A6) edge node{} (A5); \node[] at (6.5,-0.5) {$\{d+3\}$}; 
\end{tikzpicture} \par}

\caption{\label{fig:tau1}Stable tree $\tau_{1}$.}

\end{figure}
Equation (\ref{cor:int_prod})
applied to any of these stable trees is equal to $4$ (e.g., by induction
on $d$), so there are exactly $4$ maps having $\tau_{1}$ as dual
graph (Corollary \ref{cor:shape}).
The image of $\tau_{1}$ only depends on
the vertices with $2$ and $3$ leaves.
% the leaves in the vertices of valence $3$ or $4$. 
There are exactly $\binom{d+3}{3}\binom{d}{2}$ different
dual graphs with the same support as $\tau_{1}$, but with different leaves in those vertices.

If we consider another stable tree $\tau_{1}'$ in Figure \ref{fig:tau1-1},
this gives different stable maps but with the same images as $\tau_{1}$. 

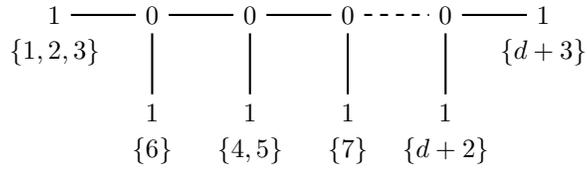
\begin{figure}[h]
{ \centering \begin {tikzpicture}[auto ,node distance =1.3 cm,on grid , thick , state/.style ={ text=black }]  \node[state] (A1){$1$}; \node[] at (0,-0.5) {$\{1,2,3\}$}; 
\node[state] (A2)[right =of A1]{$0$}; \path (A1) edge node{} (A2); \node[] at (1.3,-1.8) {$\{6\}$}; \node[state] (C2)[below =of A2]{$1$}; \path (A2) edge node{} (C2); \node[state] (A3) [right =of A2] {$0$}; \path (A2) edge node{} (A3);
\node[state] (C2) [below =of A3] {$1$}; \path (A3) edge node{} (C2); \node[] at (2.6,-1.8) {$\{4,5\}$};
\node[state] (A4) [right =of A3] {$0$}; \path (A4) edge node{} (A3);
\node[state] (T) [below =of A4] {$1$}; \path (A4) edge node{} (T); \node[] at (3.9,-1.8) {$\{7\}$};
\node[state] (A5) [right =of A4] {$0$}; \path (A4) [dashed] edge node{} (A5); %\node[dashed] at (4.5,0) {};
\node[state] (C4) [below =of A5] {$1$}; \path (A5) edge node{} (C4); \node[] at (5.2,-1.8) {$\{d+2\}$};
\node[state] (A6) [right =of A5] {$1$}; \path (A6) edge node{} (A5); \node[] at (6.5,-0.5) {$\{d+3\}$}; 
\end{tikzpicture} \par}\caption{\label{fig:tau1-1}Stable tree $\tau_{1}'$.}

\end{figure}

Since we are interested in the contact stable curves rather than in
the stable maps, we can ignore all other stable trees if they give
the same image. There is another family of graphs interesting to us,
it is $\tau_{2}$ in Figure \ref{fig:tau2}.

\begin{figure}[h]
{ \centering \begin {tikzpicture}[auto ,node distance =1.3 cm,on grid , thick , state/.style ={ text=black }]  \node[state] (A1){$1$}; \node[] at (0,-0.5) {$\{1,2\}$}; 
\node[state] (A2)[right =of A1]{$0$}; \path (A1) edge node{} (A2); \node[] at (1.3,-1.8) {$\{3,4\}$}; \node[state] (C2)[below =of A2]{$1$}; \path (A2) edge node{} (C2); \node[state] (A3) [right =of A2] {$0$}; \path (A2) edge node{} (A3);
\node[state] (C2) [below =of A3] {$1$}; \path (A3) edge node{} (C2); \node[] at (2.6,-1.8) {$\{5,6\}$};
\node[state] (A4) [right =of A3] {$0$}; \path (A4) edge node{} (A3);
\node[state] (T) [below =of A4] {$1$}; \path (A4) edge node{} (T); \node[] at (3.9,-1.8) {$\{7\}$};
\node[state] (A5) [right =of A4] {$0$}; \path (A4) [dashed] edge node{} (A5); %\node[dashed] at (4.5,0) {};
\node[state] (C4) [below =of A5] {$1$}; \path (A5) edge node{} (C4); \node[] at (5.2,-1.8) {$\{d+2\}$};
\node[state] (A6) [right =of A5] {$1$}; \path (A6) edge node{} (A5); \node[] at (6.5,-0.5) {$\{d+3\}$}; 
\end{tikzpicture} \par}

\caption{\label{fig:tau2}Stable tree $\tau_{2}$.}
\end{figure}
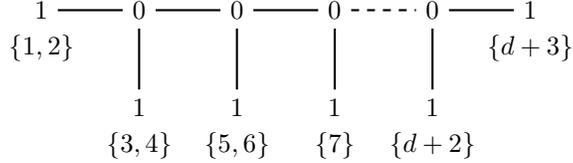

Arguing like before, we get $\frac{1}{3!}\binom{d+3}{2}\binom{d+1}{2}\binom{d-1}{2}$
different dual graphs with different images, each one with a contribution
equal to $8$. It is easy to see that there are no more interesting
stable trees, for several reasons. Each non-contracted component must
have at least one leaf by Remark \ref{rem:no_Aut}, and no more than
three leaves (there are no contact lines through $4$ general lines).
%(otherwise, (\ref{eq:int_prod}) would be $0$). 
Each
degree zero vertex $v$ must have valence $n(v)=3$ by Eq. (\ref{eq:d=00003D0}).
Moreover, it is not important the stable tree itself but how the $d+3$
leaves are distributed among the $d$ vertices of degree $1$. Hence
$\tau_{1}$ and $\tau_{2}$ are the only interesting configurations.
It follows that the number $N$ of curves we are looking for is
$$
\binom{d+3}{3}\binom{d}{2}\int_{\overline{\MC(\tau_{1})}}\prod_{i=1}^{d+3}\ev_{i}^{*}(H^{2})+\frac{1}{3!}\binom{d+3}{2}\binom{d+1}{2}\binom{d-1}{2}\int_{\overline{\MC(\tau_{2})}}\prod_{i=1}^{d+3}\ev_{i}^{*}(H^{2}).
$$
Using
$$
\int_{\overline{\MC(\tau_{1})}}\prod_{i=1}^{d+3}\ev_{i}^{*}(H^{2})=4,\,\,\,\,\,\,\,\int_{\overline{\MC(\tau_{2})}}\prod_{i=1}^{d+3}\ev_{i}^{*}(H^{2})=8,
$$
we get

\begin{eqnarray*}
N & = & \binom{d+3}{3}\binom{d}{2}4+\frac{1}{3!}\binom{d+3}{2}\binom{d+1}{2}\binom{d-1}{2}8\\
 & = & 20d\binom{d+3}{5},
\end{eqnarray*}
as expected.
\end{proof}

% \begin{eqnarray*}
% N & = & \binom{d+3}{3}\binom{d}{2}\int_{\overline{\MC(\tau_{1})}}\prod_{i=1}^{d+3}\ev_{i}^{*}(H^{2})\\
%  & + & \frac{1}{3!}\binom{d+3}{2}\binom{d+1}{2}\binom{d-1}{2}\int_{\overline{\MC(\tau_{2})}}\prod_{i=1}^{d+3}\ev_{i}^{*}(H^{2})\\
%  & = & \binom{d+3}{3}\binom{d}{2}4+\frac{1}{3!}\binom{d+3}{2}\binom{d+1}{2}\binom{d-1}{2}8\\
%  & = & 20d\binom{d+3}{5},
% \end{eqnarray*}
% as expected.
% \end{proof}

\subsection{Number of irreducible contact curves}
We denote by $N_{d}^{irr}(a_{2},\ldots,a_{n})$ the number of irreducible contact curves of degree $d$ meeting $a_i$ general linear subspaces of codimension $i$ in $\mathbb P^n$.
As application of Corollary \ref{cor:shape}, we give tables containing such numbers for $n=3,5$ and low values of $d$.
% It is enough to 
We subtract from the number of contact curves given in \cite{Mur,MurSch}, the number of reducible ones.
Using Julia (\cite{bezanson2017julia}), we implemented Eq. (\ref{eq:int_prod}) in \cite{code}. Csaba Schneider wrote the part based on the algorithm for generation of trees of \cite{McKay}.
% We implemented Eq. (\ref{eq:int_prod}) in the code \cite{code}. It is written in Julia
% % programming language 
% \cite{bezanson2017julia}. Csaba Schneider wrote the part based on the algorithm
% % for generation of trees 
% of \cite{McKay}.
% We may use Corollary \ref{cor:int_prod} to determinate the number
% of \textit{irreducible} contact curve of degree $d$ in $\mathbb{P}^{3}$ and $\mathbb{P}^{5}$
% with any Schubert condition. 
% It is enough to subtract from the number of contact curves given in \cite{Mur,MurSch}, the number of reducible ones.
% We implemented Eq. (\ref{eq:int_prod}) in the code\footnote{Csaba Schneider wrote the part of the code based on the algorithm
% of \cite{McKay}. The code is written in Julia programming language \cite{bezanson2017julia}} \cite{code}.
% We wrote a code\footnote{Csaba Schneider wrote the part of the code based on the algorithm
% of \cite{McKay}. The code is written in Julia programming language \cite{bezanson2017julia}} \cite{code} that implements (\ref{eq:int_prod}). 

\begin{rem}
The number $N_{3}^{irr}(1,3)=3$ was first computed by \cite{kal}.
The number $N_{3}^{irr}(7,0)=1080$ was first computed by \cite{Eden}
using hand-done computations of reducibles contact curves of degree
$3$. He also claimed, using the same strategy, that $N_{4}^{irr}(9,0)=378944$.
Unfortunately, his computation is not complete.

% We can compute the number $N_{2}^{irr}(a_{2},a_{3},a_{4},a_{5})$
% of smooth curves of degree $2$ in $\mathbb{P}^{5}$, using the same
% strategy.
\end{rem}

\begin{table}[h]
\centering%
\begin{tabular}{|c|c|c||c|c|c||c|c|c|}
\hline 
$d$ & $a=(a_{2},a_{3})$ & $N_{3}^{irr}(a)$ & $d$ & $a$ & $N_{4}^{irr}(a)$ & $d$ & $a$ & $N_{5}^{irr}(a)$\tabularnewline
\hline 
\hline 
\multirow{4}{*}{$3$} & $(7,0)$ & $1080$ & \multirow{5}{*}{$4$} & $(9,0)$ & $145664$ & \multirow{6}{*}{$5$} & $(11,0)$ & $65619360$\tabularnewline
\cline{2-3} \cline{3-3} \cline{5-6} \cline{6-6} \cline{8-9} \cline{9-9} 
 & $(5,1)$ & $132$ &  & $(7,1)$ & $12800$ &  & $(9,1)$ & $4501008$\tabularnewline
\cline{2-3} \cline{3-3} \cline{5-6} \cline{6-6} \cline{8-9} \cline{9-9} 
 & $(3,2)$ & $18$ &  & $(5,2)$ & $1216$ &  & $(7,2)$ & $328824$\tabularnewline
\cline{2-3} \cline{3-3} \cline{5-6} \cline{6-6} \cline{8-9} \cline{9-9} 
 & $(1,3)$ & $3$ &  & $(3,3)$ & $128$ &  & $(5,3)$ & $25884$\tabularnewline
\cline{1-3} \cline{2-3} \cline{3-3} \cline{5-6} \cline{6-6} \cline{8-9} \cline{9-9} 
\multicolumn{1}{c}{} & \multicolumn{1}{c}{} & \multicolumn{1}{c|}{} &  & $(1,4)$ & $16$ &  & $(3,4)$ & $2250$\tabularnewline
\cline{4-6} \cline{5-6} \cline{6-6} \cline{8-9} \cline{9-9} 
\multicolumn{1}{c}{} & \multicolumn{1}{c}{} & \multicolumn{1}{c}{} & \multicolumn{1}{c}{} & \multicolumn{1}{c}{} & \multicolumn{1}{c|}{} &  & $(1,5)$ & $225$\tabularnewline
\cline{7-9} \cline{8-9} \cline{9-9} 
\end{tabular}\caption{\label{tab:EnuNum}Irreducible contact curves in $\mathbb{P}^{3}$
of degree $d\le5$.}
\end{table}

\begin{table}[H]
\centering%
\begin{tabular}{|c|c||c|c||c|c|}
\hline 
$a=(a_{2},a_{3},a_{4},a_{5})$ & $N_{2}^{irr}(a)$ & $a$ & $N_{2}^{irr}(a)$ & $a$ & $N_{2}^{irr}(a)$\tabularnewline
\hline 
\hline 
$(11,0,0,0)$ & $27184$ & $(4,2,1,0)$ & $100$ & $(1,5,0,0)$ & $48$\tabularnewline
\hline 
$(9,1,0,0)$ & $7554$ & $(4,0,1,1)$ & $30$ & $(1,3,0,1)$ & $2$\tabularnewline
\hline 
$(8,0,1,0)$ & $1262$ & $(3,4,0,0)$ & $168$ & $(1,2,2,0)$ & $4$\tabularnewline
\hline 
$(7,2,0,0)$ & $2112$ & $(3,2,0,1)$ & $22$ & $(1,1,0,2)$ & $0$\tabularnewline
\hline 
$(7,0,0,1)$ & $432$ & $(3,1,2,0)$ & $16$ & $(1,0,2,1)$ & $0$\tabularnewline
\hline 
$(6,1,1,0)$ & $355$ & $(3,0,0,2)$ & $8$ & $(0,4,1,0)$ & $8$\tabularnewline
\hline 
$(5,3,0,0)$ & $594$ & $(2,3,1,0)$ & $28$ & $(0,2,1,1)$ & $0$\tabularnewline
\hline 
$(5,1,0,1)$ & $119$ & $(2,1,1,1)$ & $3$ & $(0,1,3,0)$ & $0$\tabularnewline
\hline 
$(5,0,2,0)$ & $58$ & $(2,0,3,0)$ & $2$ & $(0,0,1,2)$ & $0$\tabularnewline
\hline 
\end{tabular}\caption{\label{tab:EnuNuminP5-con}Irreducible contact conics in $\mathbb{P}^{5}$.}
\end{table}

\bibliographystyle{amsalpha}
\bibliography{RefModuliContactCurves}

\end{document}